\documentclass[10pt]{amsart}
\usepackage{amssymb}

\usepackage{graphicx}
\theoremstyle{plain}
\newtheorem{thm}{Theorem}
\newtheorem{prop}[thm]{Proposition}
\newtheorem{lem}[thm]{Lemma}

\newtheorem{cor}[thm]{Corollary}

\newcounter{claimcounter}[thm]
\newenvironment{claim}{\stepcounter{claimcounter}{{\bf Claim \theclaimcounter.}}}{}

\theoremstyle{definition}
\newtheorem{df}[thm]{Definition}
\newtheorem{pb}[thm]{Problem}
\newtheorem{exm}[thm]{Example}

\renewcommand{\leq}{\leqslant}

\newcommand{\<}{\langle}
\newcommand{\la}{\langle}
\renewcommand{\>}{\rangle}
\newcommand{\ra}{\rangle}
\newcommand{\rto}{\rightarrow}

\newcommand{\G}{\mathcal G}

\newcommand{\C}{\mathcal C}

\newcommand{\CoX}{\Co(R^n,G)}
\DeclareMathOperator{\Co}{Co}
\DeclareMathOperator{\Cld}{Cl}
\DeclareMathOperator{\Cl}{Cl}
\newcommand{\jsd}{join-sem\-i\-dis\-trib\-u\-tive}

\newcommand{\op}{\operatorname}

\begin{document}

\title{Optimum basis of finite convex geometry}
\author {K. Adaricheva}
\address{Department of Mathematical Sciences, Yeshiva University, 245 Lexington ave., New York, NY 10016}
\email{adariche@yu.edu}
\address{Department of Mathematics, School of Science and Technology, Nazarbayev University,
53 Kabanbay Batyr ave., Astana, 010000 Republic of Kazakhstan}
\email{kira.adaricheva@nu.edu.kz}

\thanks{The author was partially supported by AWM-NSF Mentor Travel grant N0839954.}

\begin{abstract}
Convex geometries form a subclass of closure systems with unique criticals, or $UC$-systems. We show that the $F$-basis introduced in \cite{AN12} for $UC$-systems, becomes optimum in convex geometries, in two essential parts of the basis: right sides (conclusions) of binary implications and left sides (premises) of non-binary ones. The right sides of non-binary implications can also be optimized, when the convex geometry either satisfies the Carousel property, or does not have $D$-cycles. The latter generalizes a result of P.L.~Hammer and A.~Kogan for acyclic Horn Boolean functions. Convex geometries of order convex subsets in a poset also have tractable optimum basis. The problem of tractability of optimum basis in convex geometries in general remains  to be open.

\end{abstract}

\maketitle
\section{Introduction}
A convex geometry is a closure system with the anti-exchange axiom.

In this paper we look at representation of \emph{finite} convex geometries by the implicational bases. This continues a series of papers \cite{ANR11} and \cite{AN12} that translate the approaches of compact presentation of finite lattices into the realm of Horn propositional logic.

If $\Sigma=\{X_i\rto Y_i: i\leq k\}$ is a set of implications defining a convex geometry, then the size of $\Sigma$ is defined as $s(\Sigma)=|X_1|+\ldots +|X_k|+|Y_1|+\ldots +|Y_k|$. The set of implications $\Sigma$ is called \emph{optimum}, when $s(\Sigma)$ is minimum among all possible sets of implications defining convex geometry.

In this paper we address the following question: \emph{ if a convex geometry is given by a set of implications $\Sigma$, is it possible to find its optimum basis $\Sigma_O$ in time polynomially dependable on $s(\Sigma)$?}

% added "for the general closure system"
 D.~Maier \cite{Mai} showed that the problem of finding the optimum basis for the \emph{general closure system}, defined by a set of implications, is NP-complete, thus, the question above most likely is answered in negative. On the other hand, some special classes of closure systems may have tractable optimum bases. These are, for example, closure systems with the modular closure lattices, as shown by M.~Wild \cite{W00}, or quasi-acyclic closure systems, as shown by P. L.~Hammer and A.~Kogan \cite{HK}. Note that the latter paper deals with Horn boolean functions and their optimal CNF-representation, and there are several variations of optimization parameters. This is further discussed in  K.~Adaricheva and J.B.~Nation \cite{AN12}.

In this paper we demonstrate three important sub-classes of convex geometries where the tractable optimum basis exists: the class of geometries satisfying the $n$-Carousel property, order convex subsets of posets and convex geometries without $D$-cycles. If the first class includes all \emph{affine} convex geometries, the third one is the generalization of \emph{acyclic} closure systems of \cite{HK}, $G$-\emph{geometries} of M. Wild \cite{W94} and (dual) supersolvable anti-matroids of D.~Armstrong \cite{Arm09}. We show that a convex geometry without $D$-cycles has the tractable optimum basis, which is exactly basis $\Sigma_{FOE}$ defined in \cite{AN12}.
We note that all three classes differ from another tractable class, $\emph{component-quadratic}$ closure systems, that generalize quasi-acyclic closure systems, see E. Boros et al \cite{BCKK}.
  
\section{Preliminaries} 
A closure system $\G=\<G, \phi\>$, i.e. a set $G$ with a closure operator $\phi :2^G\rightarrow 2^G$, is called \emph{a convex geometry} (see \cite{AGT03}), if it is a zero-closed space (i.e. $\phi (\emptyset)=\emptyset$) and it satisfies \emph{the anti-exchange axiom}, i.e.
\[
\begin{aligned}
x\in\phi(X\cup\{y\})\text{ and }x\notin X
\text{ imply that }y\notin\phi(X\cup\{x\})\\
\text{ for all }x\neq y\text{ in }G\text{ and all closed }X\subseteq G.
\end{aligned}
\]

In this paper we consider only \emph{finite} convex geometries, i.e. geometries with $|G|<\omega$.

It is worth noting that convex geometries are always \emph{standard} closure systems, i.e. they satisfy property

\[
\phi(\{i\})\setminus \{i\} \text{ is closed, for every } i \in G.
\]

This condition, in particular,  implies $i=j$, whenever  $\phi(\{i\})=\phi(\{j\})$,  for any $i,j \in G$.

Very often, a convex geometry is given by its collection of closed sets. There is a convenient description of those collections of subsets of a given finite set $G$, which are, in fact, the closed sets of a convex geometry on $G$: if $\mathcal{F} \subseteq 2^G$ satisfies\\
(1) $\emptyset \in \mathcal{F}$;\\
(2) $X\cap Y \in \mathcal{F}$, as soon as $X,Y \in \mathcal{F}$;\\
(3) $X \in \mathcal{F}$ and $X\not = G$ implies $X \cup \{a\} \in \mathcal{F}$, for some $a \in G\setminus X$,\\
then $\mathcal{F}$ represents the collection of closed sets of a convex geometry $\G=\<G,\phi\>$.

As for any closure system, the closed sets of convex geometry form a lattice, which is usually called the \emph{closure lattice} and denoted $\op{Cl}(G,\phi)$. 
The closure lattices of convex geometries have various characterizations, and are usually called \emph{ locally distributive} in the lattice literature. 

A reader can be referred to \cite{D},\cite{EdJa} and \cite{Mo1} for the further details of combinatorial and lattice-theoretical aspects of finite convex geometries.

If $Y\subseteq \phi(X)$, then this relation between subsets $X, Y \subseteq G$ in a closure system can be written in the form of implication: $X \rightarrow Y$. Thus, the closure system $\langle G,\phi\rangle$ can be given by the set of implications:
\[ 
\Sigma_\phi = \{X \rightarrow Y: X \subseteq G \text{ and } Y \subseteq \phi(X)\}.
\]

The set $X$ is called the \emph{premise}, and $Y$ the \emph{conclusion} of an implication $X\rto Y$.  We will assume that any implication $X\rto Y$ is an ordered pair of non-empty subsets $X,Y \subseteq G$, and $Y\cap X = \emptyset$.

Conversely, any set of implications $\Sigma$ defines a closure system: the closed sets are exactly subsets
$Z\subseteq G$ that \emph{respect} the implications from $\Sigma$, i.e., if $X \rightarrow Y$ is in $\Sigma$, and $X \subseteq Z$, then $Y \subseteq Z$. There are numerous ways to represent the same closure system by sets of implications; those sets of implications with some minimality property are called \emph{bases}.  Thus we can speak of various sorts of bases.

As in K.~Adaricheva et al \cite{ANR11}, we will call subset $\Sigma^b=\{(A\rto B)\in \Sigma : |A|=1\}$ 
of given basis $\Sigma$ the \emph{binary part} of the basis. Since every convex geometry $\< G, \phi\>$ is a standard closure system, the binary relation $\geq_\phi$ on $G$ defined as:
\[ a\geq_\phi b \text{  iff  } b \in \phi(\{a\})
\]
is a partial order. This is exactly the partial order of join irreducible elements in $L=\op{Cl}(G,\phi)$. If $a\geq_\phi b$, for $a\not = b$, then every basis of the closure system will contain an implication $a\rto B$ (where $b$ may or may not be in $B$). The \emph{non-binary} part of $\Sigma$ is $\Sigma^{nb}=\Sigma \setminus \Sigma^{b}$. 

We write $|\Sigma|$ for the number of implications in $\Sigma$. 
Basis $\Sigma$ is called \emph{minimum}, if $|\Sigma|\leq |\Sigma^*|$, for any other basis $\Sigma^*$ of the same system.

Number $s(\Sigma)=|X_1|+\ldots + |X_n|+|Y_1|+\ldots +|Y_n|$ is called the \emph{size} of the basis $\Sigma$. 
A basis $\Sigma$ is called \emph{optimum} if $s(\Sigma)\leq s(\Sigma^*)$, for any other basis $\Sigma^*$ of the system. Similarly, one can define $s_L(\Sigma)=|X_1|+\ldots + |X_n|$, the $L$-size, and $s_R(\Sigma)= |Y_1|+\ldots +|Y_n|$, the $R$-size, of a basis $\Sigma$. The basis will be called \emph{left-side optimum} (resp.~\emph{right-side optimum}), if 
$s_L(\Sigma)\leq s_L(\Sigma^*)$ (resp.~$s_R(\Sigma)\leq s_R(\Sigma^*)$), for any other basis $\Sigma^*$.

Now we recall the major theorem of  V.~Duquenne and J.L.~Guigues about the canonical basis \cite{DG}, also see N.~Caspard and B.~Monjardet \cite{CM03}. 

A set $Q\subseteq G$ is called \emph{quasi-closed} for $\<G,\phi\>$, if 
\begin{itemize}
\item[(1)] $Q$ is not closed;
\item[(2)] $Q\cap X$ is closed, for every closed $X$, when $Q\not \subseteq X$.
\end{itemize}
In other words, adding $Q$ to the family of $\phi$-closed sets, makes another family of sets closed stable under the set intersection, thus, a family of closed sets of some closure operator.

A quasi-closed set $C$ is called \emph{critical}, if it is minimal, with respect to the containment order, among all quasi-closed sets with the same closure. 
Equivalently, if $Q\subseteq C$ is another quasi-closed set and $\phi(Q)=\phi(C)$, then $Q=C$.

Let $\mathcal{Q}$ be the set of all quasi-closed sets and $\mathcal{C}\subseteq \mathcal{Q}$ be the set of critical  sets of the closure system $\<G,\phi\>$. 
Subsets of the form $\phi(C)$, where $C \in \mathcal{C}$, are called \emph{essential}.
%, and $\mathcal{K}$ denotes %the set of essential elements. 
It can be shown that adding all quasi-closed sets to closed sets of $\<G,\phi\>$ one obtains a family of subsets stable under the set intersection, thus, a new closure operator $\sigma$ can be defined. This closure operator associated with $\phi$ is called the \emph{saturation} operator. In other words, for every $Y\subseteq G$, $\sigma(Y)$ is the smallest set containing $Y$ which is either quasi-closed or closed.

\begin{thm}\label{DG}\textup{(}\cite{DG}, see also \cite{W94}.\textup{)}
Let $\phi$ be a closure operator on set $G$, and let $\sigma$ be saturation operator associated with it. Consider the set of implications $\Sigma_C=\{C\rto ( \phi(C)\setminus C): C \in \mathcal{C}\}$. Then
\begin{itemize}
\item[(1)]  $\Sigma_C$ is a minimum basis for $\langle G, \phi\rangle$.
\item[(2)] For every other basis $\Sigma$ of $\langle G, \phi\rangle$, for every $C\in \mathcal{C}$, there exists $(U\rto V)$ in $\Sigma$ such that $\sigma(U)=C$.
\item[(3)] Fix $C\in \mathcal{C}$ and let $\Sigma'=\{(U\rto V) \in  \Sigma_C: \phi(U)=\phi(C)\}$. Then, for any $W\subseteq C$ with $\sigma(W)=C$, the implication $W\rto \sigma(W)$ follows from $\Sigma_C\setminus \Sigma'$. 
\end{itemize}
\end{thm}

Basis $\Sigma_C$ described in Theorem \ref{DG} is called \emph{canonical}.

Some consequences can be proved from this result about the optimum basis: the premise of every implication has a \emph{fixed size} $k_C$, $C \in \mathcal{C}$, that does not depend on the choice of the optimum basis.  This makes it into a parameter of the closure system itself. Recall that a basis is \emph{non-redundant}, if none of its implications can be removed so that remaining set of implications still defines the same closure system.\footnote{Every minimum basis is non-redundant, while non-redundant basis may not be minimum.}

\begin{thm} \cite{W94} \label{W}
Let $\<G,\phi\>$ be a closure system.
\begin{itemize}
\item[(I)] If $\Sigma'$ is a non-redundant basis, then $\{\sigma(U): (U\rto V) \in \Sigma'\}\subseteq \mathcal{Q}$.
\item[(II)] Let $\Sigma_O$ be an optimum basis. For any critical set $C$, let $X_C\rto Y_C$ be an implication from this basis with $\sigma(X_C)=C$. Then $|X_C|=k_C:=min\{|U|: U\subseteq C, \phi(U)=\phi(C)\}=min\{|U|: U\subseteq C, \sigma(U)=C\}$.
\end{itemize}
\end{thm}

Another parameter of the optimum basis was found in \cite[Theorem 20]{AN12}.

\begin{thm}\label{rs-bin}
Let $\Sigma_C$ be the canonical basis of a standard closure system $\<G, \phi\>$, and let $x_C\rto Y_C$ be any binary implication from $\Sigma_C$. Every optimum basis $\Sigma$ will contain an implication $x_C\rto B$, where $|B|=b_C=\min \{|Y|: \phi(Y)=\phi(\{x_C\})\setminus\{x_C\}\}$.
\end{thm}
   
Closure systems with the \emph{unique critical sets}, or $UC$-systems, were introduced in \cite{AN12}: in such a system every essential element $X$ has exactly one critical $C\in \mathcal{C}$ with $\phi(C)=X$. 

The source of inspiration for $UC$-systems is its proper subclass of closure systems whose closure lattices satisfy the \emph{join-semidistributive law}: 
\[
(SD_\vee) \qquad
x\vee y=x\vee z \rto x\vee y= x\vee (y\wedge z).
\]

The join-semidistributive law plays an important role in lattice theory, for example in the study of free lattices, see \cite{FJN}.

It is proved in \cite[Proposition 49]{AN12} that every closure system whose closure lattice satisfies $(SD_\vee)$ is an $UC$-system. It is also well-known that $\op{Cl}(G,\phi)$ of every convex geometry $\<G,\phi\>$ is \jsd\!, see \cite{D88} and \cite{AGT03}.

Thus, convex geometries form a subclass of $UC$-systems. 
\vspace{0.3cm}

Another important subclass of $UC$-systems is represented by so-called systems \emph{without} $D$-\emph{cycles}.
The closure lattices of such systems are known in lattice literature as \emph{lower bounded}, and the lower bounded lattices form a proper subclass of \jsd\ lattices. 
%Added pointer to [5]
Since we will need the notion of the $D$-relation and the $D$-basis in section \ref{Dgeom}, we give a quick definition of these in a standard closure systems, see more details and examples in \cite{ANR11}. 

%New version of the definition of the D-basis
The set of implications $\Sigma_D = \{A\rto b\}$ for the standard closure system $\langle G, \phi\rangle$ is called the $D$-\emph{basis}, if it is a basis with the following property: for every $(A\rto b) \in \Sigma_D^{nb}$, the replacement of any $a \in A$ by any $C\subseteq  \phi(\{a\})\setminus \{a\}$ gives an implication $(A\setminus \{a\})\cup C \rto b$ which does not hold in this closure system. In particular, when $C=\emptyset$, this means that all implications in the $D$-basis have non-redundant premises.

This allows to introduce the $D$-relation: $bDa$, for some $a,b\in G$, if $a \in A$ for some $(A\rto b)\in \Sigma_D^{nb}$. The $D$-cycle is the sequence $aDa_1D\dots a_kDa$. The closure system is without $D$-cycles, if there is no sequences of such type.

Results of \cite{AN12} establish a connection between this notion and the canonical basis $\Sigma_C$, which we now outline.

Every critical set $C\in \C$ is by the definition an $\geq_\phi$-order ideal. One can find a minimal, with respect to containment, order ideal $C'\subseteq C$ such that $\phi(C')=\phi(C)$. Subset $X_K=max_{\geq_\phi}(C')$ of $\geq_\phi$-maximal elements of $C'$ is called a \emph{ minimal order generator} for essential element $\phi(C)$. Such minimal order generator is unique, if  $\op{Cl}(G,\phi)$ is \jsd\!. 
 
Given canonical basis $\Sigma_C$ of $\<G,\phi\>$, one can replace $(C\rto Y_C)\in \Sigma_C^{nb}$ by $X_K \rto Y_C$, for any minimal order generator $X_K\subseteq C$, obtaining a new basis $\Sigma'$. Now form a binary relation $\Delta_{\Sigma'}$ on $G$ as follows: $(x,y)\in  \Delta_{\Sigma'}$ iff there exists $(X\rto Y)\in \Sigma_{\Sigma'}^{nb}$ such that $x \in X$ and $y \in Y$. By $\Delta_{\Sigma'}^{tr}$ one denotes a transitive closure of relation $\Delta_{\Sigma'}$. Note that only \emph{non-binary} implications participate in definition of $\Delta_{\Sigma'}$.

\begin{thm}\cite{AN12} 
A standard closure system $\<G,\phi\>$ is without $D$-cycles iff $\Delta_{\Sigma'}$ does not have cycles, i.e. $(x,x) \not \in \Delta_{\Sigma'}^{tr}$.
\end{thm}

In section \ref{Dgeom} we will also need a definition of a $K$-basis of a standard system.

\begin{df}\cite{AN12}
Set of implications $\Sigma_K$ is called a $K$-\emph{basis}, if it is obtained from canonical basis $\Sigma_C$ by replacing each implication $(C\rto Y_C)\in\Sigma_C$ by $X_K\rto Y_K$, where $X_K\subseteq C$ is a minimal order generator of $\phi(C)$, and $Y_K=\op{max}_{\geq_\phi} (Y_C)$.
\end{df}
In particular, by Theorem \ref{DG}, a $K$-basis is minimum and  $s(\Sigma_K)\leq s(\Sigma_C)$.
Note that if $C\rto Y_C$ is in $\Sigma^b_C$, i.e. $C=\{x\}$, for some $x \in G$, then $X_K=C=\{x\}$.

Further modifications of the $K$-basis exist for join-semidistributive closure systems and for closure systems without $D$-cycles. In the first case, the binary part is replaced by implications $x \rto Y$, where $Y$ is a unique minimal order generator of closed set $X_*=\phi(\{x\})\setminus \{x\}$. Such basis is called $F$-basis in \cite[Definition 52]{AN12}. 

For the second class, the $E$-basis is an optimization of the $K$-basis in the right sides of its non-binary implications. Namely, if $X\rto Y_1$ and $Z\rto Y_2$ are in the non-binary part of the  $K$-basis, where some $y \in Y_1\cap Y_2$ and $\phi(X)\subset \phi(Z)$, then $y$ will be excluded from $Y_2$ in corresponding implication of the $E$-basis. This modification is not possible in general closure systems, thus, the absence of the $D$-cycles is an essential pre-requisite. 

Further reduction of the right sides of non-binary part of the $E$-basis is possible, if elements in the right side of some implication are comparable by $\geq_\phi$ relation. Say, if $x\geq_\phi y$, then keeping $y$ in the right side is not necessary, since $x\rto y$ follows from the binary part of the basis. Thus, one can keep in the right sides only $\geq_\phi$-maximal elements. This further reduction of the $E$-basis is called in \cite{AN12} the \emph{optimized} $E$-basis, or $\Sigma_{OE}$.

Finally, since closure systems without $D$-cycles are join-semiditsributive, both modifications of the $K$-basis in its binary part (as in $F$-basis) and its non-binary part (as in optimized $E$-basis) gives basis $\Sigma_{FOE}$ for such systems.
\vspace{0.3cm}

\section{Convex geometries}\label{CG}

In this section we make the general observations about the bases of convex geometries.

\begin{lem}\cite{D88}
If $\G=\<G,\phi\>$ is a finite convex geometry, then $\op{Cl}(G,\phi)$ is join-semidistributive.
\end{lem}

According to \cite[Proposition 41]{AN12}, every closure system with \jsd\ closure lattice has the unique $K$-basis.

Recall that the set of extreme points of a closed set $X\subseteq G$ is defined as $Ex(X)=\{x \in X: x \not \in \phi(X\setminus \{x\})\}$.  It is well-known that, in every convex geometry, for every closed set $X$, $X=\phi(Ex(X))$, see \cite{EdJa}. The equivalent statement in the framework of lattice theory is that the closure lattice of a finite convex geometry has unique irredundant join decompositions; see, for example, \cite[Theorem 1.7]{AGT03}. The closure lattices of finite convex geometries are known in the literature as \emph{locally distributive}, or \emph{meet distributive}. Such lattices $L$ are characterized by the property that, for every element $x \in L$, if $y=\bigwedge \{x' \in L: x' \prec x\}$, then the interval $[y,x]$ is Boolean.

The following statement was observed in \cite{W94}, Corollary 13(b). Recall from Theorem \ref{W} (II) that every optimum basis of any closure system has an implication $X_C\rto Y_C$, corresponding to a critical set $C$, with $|X_C|=k_C$. 

\begin{thm}\label{conGopt}  If\/ $\G=\< G, \phi\>$ is a convex geometry, then the $K$-basis is left-optimum, and for every critical set $C$, the corresponding implication $X_C\rto Y_C$ in the $K$-basis satisfies $X_C=Ex(\phi(C))$.  

\end{thm}

\begin{proof}
If $X=\phi(C)$ is an essential (closed) element of the closure system, $Ex(X)=Ex(C)$ is the premise of implication in the $K$-basis, corresponding to $X$. Since $Ex(X)$ is the unique irredundant generator for $X$, it should also appear as a premise in every optimum basis for $\G$. 
\end{proof}

Recall that the closure system $\G=\< G, \phi\>$ is called \emph{atomistic}, if $\phi(\{x\})=\{x\}$, for every $x \in G$.
\begin{cor}\label{left-opt}

Every $K$-basis of an atomistic \jsd\ closure system is left-side optimum.

\end{cor}

Indeed, this follows from Theorem \ref{conGopt} and Corollary 1.10 in \cite{AGT03}, that states that every atomistic \jsd\ closure system is a convex geometry.

We also observe that the binary part of any optimum basis of any convex geometry is uniquely defined.
Recall that basis $\Sigma$ of any standard closure system was called \emph{regular} in \cite{AN12}, if for every $(x\rto B)\in \Sigma^b$, it holds $\phi(B)=\phi(\{x\})\setminus\{x\}$. It was shown in \cite[Corollary 17]{AN12} that every optimum basis of a standard closure system is regular.

\begin{lem}\label{RSconvgeo}

If $\Sigma$ is a (regular right-side) optimum basis of a convex geometry, then, for every $(x \rto Y)\in \Sigma$, $Y=Ex (\phi(\{x\})\setminus \{x\})$.

\end{lem}

\begin{proof}

According to Theorem 16 in \cite{AN12}, for every $x \rto Y$ in $\Sigma$, $Y$ is the set of minimal cardinality with the property 
$\phi(Y)=X_*=\phi(\{x\})\setminus \{x\}$. Moreover, according to Corollary 18 in \cite{AN12}, $Ex(X_*) \subseteq Y$. Hence, $Y=Ex(X_*)$, and such conclusion in any optimum basis is unique.
\end{proof}

We note that in terminology of \cite{AN12}, set $Y=Ex(X_*)$ in the proof of preceding Lemma is simultaneously the \emph{minimal order generator} for closed set $X_*$, and such generators are unique in closure systems with \jsd\ closure lattices. The basis $\Sigma$ of any \jsd\ system, whose binary part comprises $x\rto Y$, where $Y$ is a unique order generator of closed set $X_*= \phi(\{x\})\setminus \{x\}$ is called $F$-basis in \cite[Definition 52]{AN12}.

The non-binary part of the $F$-basis is the same as in $K$-basis. The $F$-basis has the further refinement in the systems without $D$-cycles, and we will return to it in section \ref{Dgeom}. 

\section{Convex geometries with the Carousel property}

An important example of a (finite) convex geometry is $\CoX$, where $G$ is a (finite) set of points in $R^n$, and $\CoX$ stands for the geometry of convex sets relative to $G$. In other words, the base set of such closure system is $G$, and closed sets are subsets $X$ of $G$ with the property that whenever point $x \in G$ is in the convex hull of some points from $X$, then $x$ must be in $X$ (see more details of the definition, for example, in \cite{AGT03}).
We will call convex geometries of the form $\CoX$ \emph{affine}.

The following definition is a slight modification of the property introduced in \cite{A08}.

\begin{df}\label{nCar}

\item A closure system $\G=\<G,\phi\>$ satisfies the $n$-Carousel property, if for every $X \subseteq G$, that has at least two elements, and $x,y \in \phi(X)$, there exists $X'\subset X$ such that $|X'|\leq \min\{n,|X|-1\}$ and  $x \in \phi(\{y\} \cup X')$.

\end{df}

The $2$-Carousel property was an essential tool in dealing with representation problem for affine convex geometries in K.~Adaricheva and M.~Wild \cite{AW10}.

If a closure system $\mathcal{G}=\la G,\phi\ra$ satisfies the $n$-Carousel property, then, assuming that $y$ may be taken in $X$, we see that the closures in $\mathcal{G}$ are fully defined by the closures of at most $(n+1)$-element subsets of $X$. In particular, $\mathcal{G}$ also satisfies the $n$-\emph{Carath\'eodory property}: 

if $x \in \phi(Y)$, $Y \subseteq X$, then $x \in \phi(x_0,\dots,x_n)$ for some $x_0,\dots,x_n \in Y$.

The following statement follows from \cite[Lemma 2.3]{A08}.

\begin{lem}\label{good} Every convex geometry $\CoX$, where $G$ is a finite set of points in $R^n$, satisfies the $n$-Carousel property.

\end{lem}

\begin{thm}\label{opt-nCar}

If $\G=\<G,\phi\>$ is any convex geometry satisfying the $n$-Carousel property, then one can obtain an optimum basis in time $O(|\Sigma_C|^2)$.

\end{thm}
\begin{proof}
Let $\Sigma_C=\{C\rto \phi(C)\setminus C: C \in \mathcal{C}\}$ be the canonical basis of $\mathcal{G}$. We know from the proof of Theorem \ref{conGopt} that the set of implications $\Sigma_{ex}=\{Ex(C)\rto \phi(C)\setminus C: C \in \mathcal{C}\}$ is also a basis of $\mathcal{G}$.

We now write a new set of implications $\Sigma$:
\begin{itemize}
\item for each non-binary implication $Ex(C)\rto \phi(C)\setminus C$ in $\Sigma_{ex}$, pick any $b \in \phi(C)\setminus Ex(C)$, and replace this implication by $Ex(C)\rto b$;
\item replace each binary implication $a \rto B$ in $\Sigma_{ex}$ by $a \rto Ex(B)$.
\end{itemize}

We need to show that $\Sigma$ is also the basis for $\mathcal{G}$. For this, we associate with $\Sigma$ closure operator $\tau$ and show that every set $Y\subseteq G$ is $\phi$-closed iff it is $\tau$-closed.

Note that $\Sigma$ only reduces the conclusions in implications of $\Sigma_{ex}$. Hence, $\tau(Y)\subseteq \phi(Y)$, for every $Y\subseteq G$. In particular, every $\phi$-closed set is $\tau$-closed.
Also, since $\la G, \phi\ra$ is standard, $\la G, \tau\ra$ must be standard as well. For this, we observe that  $\tau(\{a\})\setminus \{a\} = \tau(\{a\}) \cap  (\phi(\{a\})\setminus \{a\})$ must be $\tau$-closed, since $\phi(\{a\})\setminus \{a\}$ is $\phi$-closed and every $\phi$-closed set is $\tau$-closed.

So now we consider any $\tau$-closed set $Z$, and argue by induction on the height of $Z$ in the closure lattice $\Cl(X,\tau)$. 

The least $\tau$-closed set is $\emptyset$, which is also $\phi$-closed.

Now assume that $Z$ is some $\tau$-closed set, and it has already been shown that every $\tau$-closed $Z' \subset Z$ is also $\phi$-closed. In what proceeds, we will show that $Z$ is also $\phi$-closed. First, it is done in case when $Z$ is join irreducible in $\Cl(X,\tau)$. Then we turn to case when $Z$ is not join irreducible, which in turn splits into two cases: when $\phi(Z)$ is essential element in $\Cl(X,\phi)$ and when it is not.

\vspace{0.5cm}
\begin{claim}
If $Y=\tau(\{a\})\subseteq Z$, then $Y=\phi(\{a\})$.
\end{claim}

\begin{proof}
Since $Y_*=Y\setminus \{a\}$ is $\tau$-closed, it is also $\phi$-closed, by inductive assumption. If $(a\rto B)\in \Sigma_{ex}$, then $B=\phi(\{a\})\setminus \{a\}$, and, due to $\tau(\{a\})\subseteq \phi(\{a\})$ we have $Y_*\subseteq B$. On the other hand, $Ex(B)\subseteq Y_*$ due to implication $a\rto Ex(B)$ in $\Sigma$, hence, $B=\phi(Ex(B))\subseteq Y_*$. Therefore, $B=Y_*$ and $\phi(\{a\})=B\cup \{a\}=Y_*\cup \{a\}=Y$ is $\phi$-closed. 
\end{proof}

If $Z$ is a join-irreducible in $\Cld(X,\tau)$, then $Z=\tau(\{a\})$, for some $a \in X$.  Applying Claim 1, we obtain that $Z$ is $\phi$-closed.

Now assume that $Z$ is join reducible in $\Cld(X,\tau)$. First we want to show that $\phi(Z)$ is join reducible in $\Cld(X,\phi)$.

Suppose $Z_1=\phi(Z)$ is join irreducible in $\Cld(X,\phi)$. Then $Z_1=\phi(\{a\})$, for some $a \in X$. If $a \not \in Z$, then $\phi(\{a\})\setminus \{a\}$ is not $\phi$-closed: we would have $Z\subseteq \phi(\{a\})\setminus \{a\}$, but $\phi(Z)\not \subseteq \phi(\{a\})\setminus \{a\}$. This contradicts to the fact that $\la X, \phi\ra$ is a standard closure system. Hence, $a \in Z$. 

Consider $\tau(\{a\})\subseteq Z$.  Applying Claim 1, conclude that $\tau(\{a\})=\phi(\{a\})=Z$. This will contradict the assumption that $Z$ is join reducible in $\Cld(X,\tau)$.

Thus, $Z_1=\phi(Z)$ must be join reducible.

(1) First, consider the case when $Z_1$ is essential element in $\mathcal{G}$. Then there exists $(C\rto \phi(C))\in \Sigma_C$ such that $Z_1=\phi(C)$, $|C|>1$, hence, $(Ex(C)\rto \phi(C)) \in \Sigma_{ex}$, $|Ex(C)|>1$. Apparently, $Ex(C)\subseteq Z$. This implies $b\in Z$, where $(Ex(C)\rto b)\in \Sigma$.   

%Since $\mathcal{G}$ satisfies $n$-Carath\'eodory property, $|A'|\leq n+1$.

Now we want to apply the $n$-Carousel property to show that every $b' \in \phi(C)$ belongs to $Z$. We have $b',b \in \phi(Ex(C))$, then $b' \in \phi(A\cup \{b\})$, for some $A\subset Ex(C)$. In particular, $A$ misses an extreme element of $Z_1$, hence, $\phi(A\cup \{b\}) \subset Z_1$. 

We have $\tau(A\cup \{b\})\subset Z$,  otherwise $\phi(A\cup \{b\})=\phi(Z)=Z_1$, a contradiction. According to the inductive assumption, $\tau(A\cup \{b\})$ is also $\phi$-closed. This implies $b' \in \tau(A\cup \{b\})\subseteq Z$, as desired.

(2) Secondly, consider the case when $Z_1$ is not essential in $\mathcal{G}$. Take $A=Ex(Z_1)$, noting that $A\subseteq Z$. The implication $A\rto Z_1\setminus A$ follows from the basis $\Sigma_{ex}$. In particular, for every $z \in Z_1\setminus A$, there is a sequence $\sigma_1,\dots, \sigma_t$ of implications from $\Sigma_{ex}$, with $\sigma_k=(A_k\rto B_k)$, such that $A_1\subseteq A$, $z \in B_t$ and $A_k \subseteq A\cup B_1\cup\dots \cup B_{k-1}$, $t\geq k>1$.

If $\tau(A_1)=Z$, then $\phi(A_1)=\phi(Z)=Z_1$, which contradicts to $Z_1$ being not essential. Hence, $\tau(A_1)\subset Z$, and according to inductive assumption, $\tau(A_1)=\phi(A_1)$, so that $B_1\subseteq Z$. This implies that $A_2\subseteq Z$, and by a similar argument, we conclude that $B_2\subseteq Z$. Proceeding along the sequence $\sigma_1,\dots, \sigma_t$, we obtain eventually, that $z \in Z$. Hence, $Z_1\subseteq Z$, and $Z$ is $\phi$-closed.

This finishes the proof that $\Sigma$ is a basis for $\mathcal{G}$. It follows that $\Sigma$ is an optimum basis. Indeed, it is left-side optimum due to Theorem \ref{conGopt}. For the right sides, it cannot be made shorter for non-binary implications. For the binary implications, the right-side optimality follows from Lemma \ref{RSconvgeo}.

\end{proof}

\begin{cor}\label{one} For every optimum basis $\Sigma_O$ of an affine convex geometry, for every $(A\rto B)\in \Sigma_O^{nb}$, $|B|=1$.
\end{cor}
\begin{proof}  First, we point that $R_O^{nb}=|B_1|+\dots +|B_k|$ is  a fixed parameter for any given closure system, where $B_i$, $i\leq k$, are the right sides of all implications in the non-binary part of the optimum basis. Indeed, it follows from Theorems \ref{W}(II) that the total size $L_O^{nb}$ of left sides of implications from the non-binary part of any optimum basis is a fixed parameter, and it follows from Theorem  \ref{rs-bin} that the same is true for the total size $R_O^b$ of right sides of the binary part. The total size $L_O^b$ of left sides of the binary part is also fixed, since it is given by the number of implications in the binary part. $R_O^{nb}$ complements $L_O^{nb}+R_O^b+L_O^b$ to the full size of the optimum basis, from which the observation follows.

It is proved in Theorem \ref{opt-nCar} that every affine convex geometry has $R_O^{nb}=k$, where $k$ is the number of implications in the non-binary part of the canonical basis. Hence, every other optimum basis should have one-element conclusions in its non-binary part.

\end{proof}

Firstly, we note that the geometries with $n$-Carousel property include the class $\CoX$, due to Lemma \ref{good}, but they are not reduced to this class. The result of Theorem \ref{opt-nCar}  for the class 
$\CoX$ was also proved in K.~Kashiwabara and M.~Nakamura \cite{N12}.

\begin{exm}
Consider a convex geometry $\G$ defined by the canonical basis $\Sigma_C=\{abc\rto xz, acx\rto z, z\rto x\}$. Apparently, this geometry satisfies the $2$-Carousel property, but it cannot be represented as $\Co (R^2,G)$, because the latter geometry is atomistic, while $\G$ has the binary implication $z\rto x$. According to Theorem \ref{opt-nCar}, an optimum basis of this geometry is either of the following two: $\{abc\rto z, acx\rto z, z\rto x\}$, or $\{abc\rto x, acx\rto z, z\rto x\}$.\\
\end{exm}

Secondly, we note that the $n$-Carousel property in Definition \ref{nCar} is stronger than the version introduced in \cite{A08}. In particular, the result of \cite{A08} that every subgeometry of the geometry with the $n$-Carousel property satisfies this property is no longer true under the new definition. This happens because a subgeometry of the geometry with the Carath\'eodory number $n$ may have Carath\'eodory number $<n$. This is illustrated in the following example.

\begin{exm}
Consider $5$-point configuration $A=\{a,b,c,x,z\}$ on a plane $R^2$, where $a,b,c$ form a triangle with points $x,z$ inside, so that $x$ is also in triangle $abz$, and $z$ is in triangle $acx$. Then the canonical basis of convex geometry $\G=\Co (R^2,A)$ is $\Sigma_C=\{abc\rto xz, acx\rto z, abz\rto x\}$. According to Theorem \ref{opt-nCar}, the optimum basis will be either of two: $\{abc\rto z, acx\rto z, abz\rto x\}$ or $\{abc\rto x, acx\rto z, abz\rto x\}$.

Now consider the geometry $\G_1$ defined on $A$ by the following implications $\Sigma = \{a\rto c, ab\rto xz,ax\rto z\}$. In fact, one can verify that $\G_1$ is obtained from $\G$ by adding the implication $a\rto c$. Moreover, the closure lattice of $\G_1$ is a sublattice of closure lattice of $\G$. Thus, $\G_1$ is a sub-geometry of $\G$, in terminology of \cite{A08}.

While geometry $\G$ satisfied $3$-Carath\'eodory and $3$-Carousel property, $\G_1$ has the stronger $2$-Carath\'eodory property. In the old definition of \cite{A08}, $\G_1$ still satisfies $3$-Carousel property, which is in this case simply equivalent to $2$-Carath\'eodory property. But $\G_1$ fails the $3$-Carousel under Definition \ref{nCar}, since $x,z \in \phi(\{a,b\})$ in $\G_1$, while $x \not \in \phi(\{z\}\cup A')$, for any proper subset $A'\subset \{a,b\}$.

\end{exm}

Thirdly, we note that geometries of the form $\CoX$ are an essential source of closure systems outside the $CQ$-class of Boolean functions, for which an optimum basis can be effectively found, as shown in \cite{BCKK}. According to definition, a closure system (a Horn Boolean function) $\<G,\phi\>$ is $CQ$, or \emph{component quadratic}, if it has basis $\Sigma=\{A_C\rto B_C: C \in \C\}$ such that $A_C$ has no more than one element from $\Sigma$-component of $b$, for every $b \in B_C$. By a $\Sigma$-component of element $b$ we mean all elements $b' \in X$ such that $b\rto^\Sigma b'$ and $b'\rto^\Sigma b$. Here $b\rto^\Sigma b'$ means that $(b,b')$ is in the transitive closure of the relation 

$\Box_\Sigma=\{(x,y)\in X^2: x \in A_C, y \in B_C, (A_C\rto B_C)\in \Sigma\}$.

\begin{figure}[htbp]\label{pic1}

\begin{center}

\includegraphics[height=2in,width=3in]{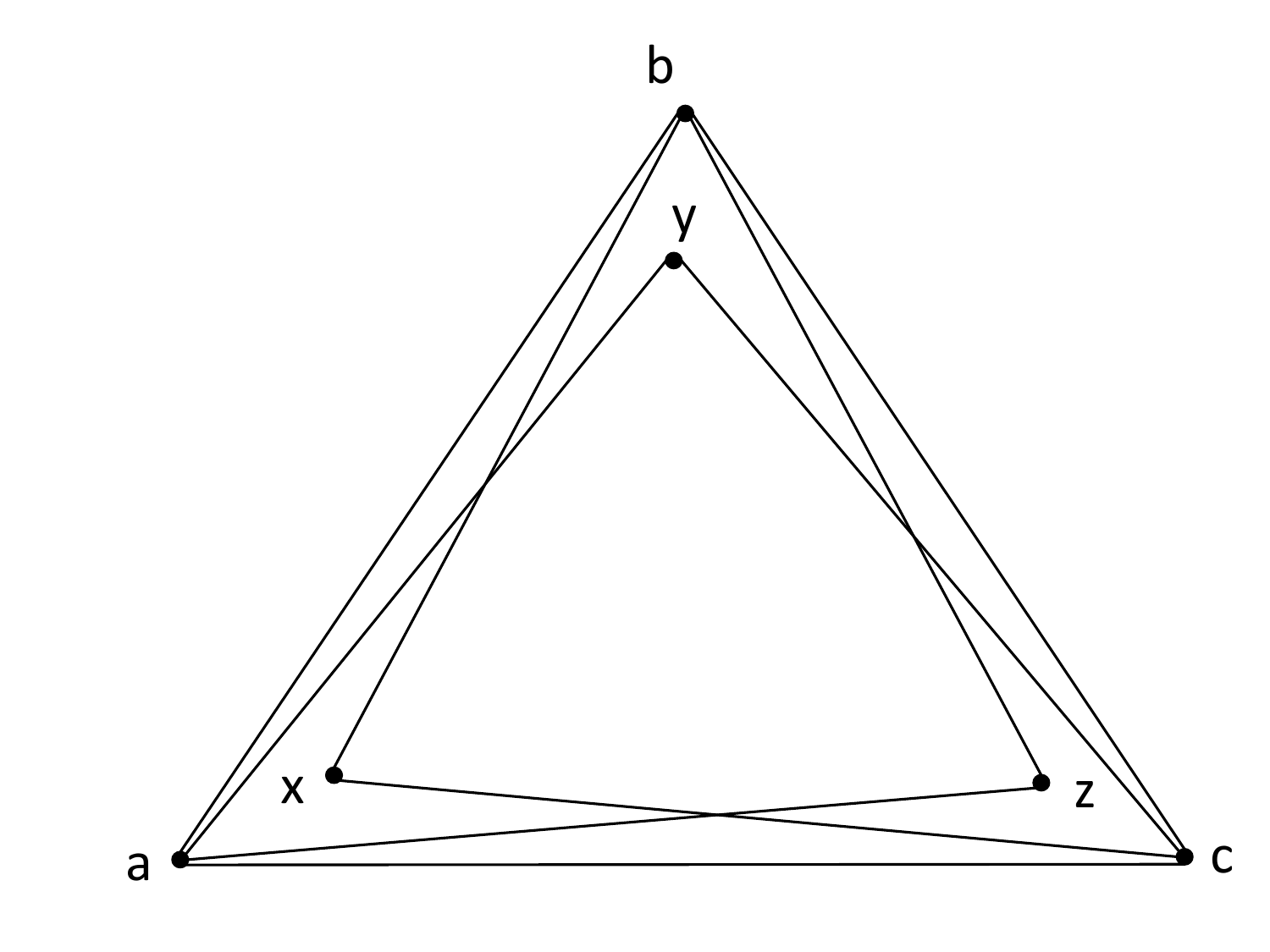}

\caption{Example \ref{notQC}}

\end{center}

\end{figure}

\begin{exm}\label{notQC}
Consider 6-point configuration $G=\{a,b,c,x,y,z\}$ in $R^2$ given on Figure \ref{pic1}, where $x,y,z$ are inside triangle $abc$. Convex geometry $\G=\Co (R^2,G)$ is given by the following canonical basis: $\Sigma_C=\{abc\rto xyz, abz\rto xy, acy \rto xz, bcx \rto yz, ayz\rto x, bxz\rto y, cxy\rto z\}$. 

According to Theorem \ref{opt-nCar} and Corollary \ref{left-opt}, any optimum basis for $\G$ will have the same premises as $\Sigma_C$ and will contain the implications $ayz\rto x, bxz\rto y, cxy\rto z$. This implies that $x,y,z$ are in the same $\Sigma$-component, for every optimum basis. On the other hand, each of these three implications have two elements from this component in the premise.

\end{exm}

%We will consider another subclass of convex geometries with the tractable optimum basis in the next section.

\section{Convex geometries without $D$-cycles}\label{Dgeom}

In this section we establish that another subclass of convex geometries has tractable optimum bases.

\begin{df}
Call a closure system $\<G,\phi\>$ a $D$-\emph{geometry}, if it is a convex geometry and does not have $D$-cycles.
\end{df}

\begin{prop}
A closure system is a $D$-geometry iff its closure lattice is meet distributive and lower bounded.
\end{prop}

One important subclass of $D$-convex geometries was considered in \cite{HK} under the name \emph{acyclic Horn Boolean functions}\footnote{Note that within the class of convex geometries, acyclic and quasi-acyclic Horn Boolean functions are equivalent concepts.} and in \cite{W94}, under the name $G$-\emph{geometries}.

Essentially, both can be defined as follows. Let $(P,\leq)$ be any partially ordered set. Define a closure system on $P$ by any set of implications $\Sigma=\{A_k\rto B_k: k\leq n\}$ so that for every $a \in A_k$ and $b \in B_k$ we have $b \leq a$. Following L.~Santocanale and F.~Wehrung \cite{SW14},  a closure operaotr defined via such a set of implications will be called \emph{of poset type}, and we say that the implications of $\Sigma$ are \emph{compatible} with $(P,\leq)$.

\begin{lem}\label{poset-def}

Let $A\rto B$ be any implication from the $D$-basis for an operator of poset type, with respect to some poset $(P,\leq)$. Then $b\leq a$, for every $a \in A$ and $b \in B$. In particular, $\la P,D^{tr}\ra$, where $D^{tr}$ is the transitive closure of the $D$-relation, is a sub-poset in $(P,\leq)$.

\end{lem}

\begin{proof}

It is straightforward to show that the canonical direct unit basis, which is a center of discussion in K.~Bertet and B.~Monjardet \cite{BM}, is compatible with the poset $(P,\leq)$. Indeed, just use Proposition 4 and Theorem 15 from \cite{BM}. Since  the $D$-basis is a subset of the canonical direct unit basis, see \cite[Lemma 8]{ANR11}, we get the desired conclusion.

\end{proof}

\begin{cor}\label{D-geom}

Every closure system $\<P,\phi\>$ of poset type is a $D$-geometry.

\end{cor}

\begin{proof}

Indeed, it follows from Lemma \ref{poset-def} that $\<P,\phi\>$ does not have $D$-cycles. It is easy to check also that $y \in \phi(X)$, $X\subseteq P$, implies $y \in \phi(X')$, for some $X'\subseteq X$ such that $y\leq x'$, for all $x'\in X'$. From this, the anti-exchange property of convex geometry directly follows.  

\end{proof}

It was observed in M. Wild \cite{W94} that closure systems of poset type (called there as $G$-geometries) are convex geometries. Corollary 15 in the same paper also established that all optimum implicational bases of $G$-geometries have no directed cycles. 

In the terminology of \cite{HK}, also given at the end of section \ref{CG}, prior to Example \ref{notQC}, this is equivalent to say that in such a system, every $\Sigma$-component of any optimum basis $\Sigma$ consists of a single element. 

Thus, Corollary \ref{D-geom} implies that acyclic Horn Boolean functions of \cite{HK} and $G$-geometries of \cite{W94} are $D$-geometries.

On the other hand, there exist $D$-geometries which are not of poset type.

\begin{exm}\label{notPoset-def}
Consider a closure system defined by its optimum basis $\Sigma = \{a_1a_2\rto b_1, b_1b_2\rto c_1,c_1c_2\rto d, c_1\rto a_1, b_2\rto a_1, d\rto a_2, c_2\rto a_2\}$. It is straightforward to check that the closure system defined by $\Sigma$ is a convex geometry, and examining the non-binary part, one does not find $D$-cycles. Hence, it is a $D$-geometry. On the other hand, this system has a non-trivial component $\{a_1,b_1,c_1,d,a_2\}$, and thus it cannot be of the poset type.

Moreover, the first implication has two elements from the $\Sigma$-component of $b_1$, so this is not a $CQ$-system.

\end{exm}
%Added conclusion about generalization.
Combination of Corollary \ref{D-geom} and Example \ref{notPoset-def} gives the following statement.

\begin{cor} Closure systems of poset type ($G$-geometries) form a proper subclass of $D$-geometries.
\end{cor}

The next statement combines results of \cite{AN12} and section \ref{CG}. We need to recall the definition of basis $\Sigma_{FOE}$ introduced for closure systems without $D$-cycles in \cite[Definition 70]{AN12}. Letter ``$F$" in the notation comes from the $F$-basis, since the binary part of $\Sigma_{FOE}$ is defined as in $F$-basis, see the end of section \ref{CG}. Thus, if $(x\rto Y) \in \Sigma_{FOE}$, then $\phi(Y)=\phi(\{x\})\setminus \{x\}$.

Letters ``$OE$" in the notation come from ``optimized $E$-basis". The $E$-basis was defined in \cite{ANR11}, for the systems without $D$-cycles, and it was further analyzed in \cite{AN12}, for its connection with the canonical basis. The non-binary part of the $E$-basis has implications $X_K\rto Y_O$, where $X_K$ is defined as in the $K$-basis, i.e. $X_K\subseteq C$ is a minimal order generator of essential element $\phi(C)$, for some $C \in \mathcal{C}$. The conclusion $Y_O\subseteq Y_K$, is a subset of $Y_K$, the right side in the $K$-basis. Element $y \in Y_K$ is included in $Y_O\subseteq Y_K$, only if there is no other $C'\in \mathcal{C}$, $|C'|>1$, such that $y\in \phi(C')\setminus C'$ and $\phi(C')\subset \phi(C)$. 

\begin{thm}\label{optDgeom}

If $\<G,\phi\>$ is a $D$-geometry, then its basis $\Sigma_{FOE}$ is optimum.

\end{thm}

\begin{proof}

The premises of $\Sigma_{FOE}$ and the $K$-basis coincide by the definition. Since $D$-geometry is a convex geometry, one can apply Theorem \ref{conGopt} to claim that $\Sigma_{FOE}$ is left-side optimum.

The right sides of the binary implications are optimum due to Lemma \ref{RSconvgeo}.

%Since $D$-geometry is a closure system without $D$-cycles, the right sides of the non-binary part of $\Sigma_{FOE}$ %are the same as in optimized $E$-basis.

Corollary 57 in \cite{AN12} shows that $\Sigma_{FOE}$ is also optimum in its non-binary right side. This implies that $\Sigma_{FOE}$ is left-side optimum and right-side optimum, whence it is optimum. 

\end{proof}

It was shown in \cite[Lemma 62]{AN12} that, for any closure system $\< G, \phi\>$ and its canonical basis $\Sigma_C$, it requires time $O(s^2(\Sigma_C))$ to recognize whether the system is without $D$-cycles and obtain its $\Sigma_{OE}$ basis. Obtaining binary implications of $F$-basis amounts to finding a (unique) minimal order generator of closed sets $\phi(\{x\})\setminus\{x\}$, for each $x \in G$. This can be done at most in time $O(|G|^2)$, similarly to algorithm of \cite[Proposition 30]{AN12}. Thus, the following can be deduced from Theorem \ref{optDgeom}.

\begin{cor} The optimum basis of every $D$-geometry can be obtained from its (canonical) basis in time polynomial in size of the basis.
\end{cor} 

We can mention two well-known subclasses of convex geometries without $D$-cycles.

The first contains $\op{Sub}_\wedge(S)$, the convex geometries of subsemilattices of a $\wedge$-semilattice $S$, where the canonical basis is given by $\{ab\rto c: a\wedge b=c, a,b,c\in S\}$. It was proved in \cite{A96} that finite lattices $\op{Sub}_\wedge(S)$ are lower bounded. Moreover, they are atomistic, which guarantees (with the addition of the \jsd\ law) that they are convex geometries, see \cite{AGT03}.

Similarly, the lattice $\op{O}(P)$ of suborders of a partially ordered set $\la P,\leq\ra$ is lower bounded, by result of Sivak \cite{Si78}. It gives the closure lattice of a convex geometry defined on set $X = \{(a,b)\in P^2: a<b\}$. The canonical basis in this case is $\Sigma_C=\{ab\rto c: a=(x,y),b=(y,z),c=(x,z) \in X\}$.

In both cases, the canonical basis cannot be refined, so it is already optimum.\\

In conclusion of this section we also mention the connection between lattices of the poset type and supersolvable lattices.

Supersolvable lattices were introduced by R. Stanley in \cite{S72}. The motivating examples were lattices of subgroups of supersolvable finite groups.

\begin{df}\label{supersolv}
A maximal chain in lattice $L$ of finite height is called an $M$-chain, if together with any other maximal chain it generates a distributive sublattice in $L$. Lattice is called \emph{supersolvable}, if it has an $M$-chain.
\end{df}

The key combinatorial description of supersolvable lattices was given in P. McNamara \cite{McN03}.

It was shown in K.~Adaricheva \cite{A13} that every supersolvable and \jsd\ lattice must be meet distributive, i.e. it must be a closure lattice of a convex geometry. Moreover, it was observed in K. Kashiwabara and M. Nakamura \cite{KN12}, based on work of D. Armstrong \cite{Arm09}, that convex geometry is supersolvable iff it is of the poset type. The combination of these two results gives a full description of \jsd\ supersolvable lattices.

\section{Convex geometries of order convex subsets}

Let $\<P,\leq\>$ be a partially ordered set.
Denote $\Co(P)$ convex geometry $\<P,\phi\>$, where $\phi(X)$ is a smallest convex subset of $P$ containing $X\subseteq P$. By the definition, a subset $Y\subseteq P$ is \emph{convex}, if $a\leq c\leq b$ and $a,b\in Y$ implies $c \in Y$.

It is easy to verify that the canonical basis of any convex geometry $\Co(P)$ does not have a binary part and comprises implications $xy\rto Z$, where $x<y$ in $P$, and $Z=[x,y]=\{z\in P: x<z<y\}$. Similarly, it is easy to describe implications of the $D$-basis: $xy \rto z$, for some $x<z<y$.

%added connection to two other classes
We observe that the this subclass of convex geometries is rather disjoint from two others we discussed so far.
It is more often than not convex geometry $\Co(P)$ contains $D$-cycles. Indeed, whenever we have $4$-element chain $a<b<c<d$ in $\<P,\leq\>$, we will get two implications in the $D$-basis: $bd\rto c$ and $ac\rto b$. Thus, $bDcDb$ is a $D$-cycle. 

Similarly, it is rare that $\Co(P)$ satisfies $n$-carousel property, for any $n$. If  $a<b,c < d$ and $b,c$ are incomparable, then $ad\rto bc$, but none of next implication holds: $ab\rto c$, $db\rto c$, $ac\rto b$, $dc\rto b$, i.e. $1$-carousel property fails.

The canonical basis of $\Co(P)$ is already left-side optimized. Thus, the task of optimizing the basis is to choose, for every implication $xy\rto Z$, a subset $Z'\subseteq Z$ so that implication $xy\rto Z'$ will belong to an optimum basis. 

For every $a<b$ in poset $\<P,\leq\>$, let denote $Cp[a,b]$ the number of connected components of sub-poset on $[a,b]\setminus\{a,b\}$. Connected component of any poset is defined as connected component of the graph of the cover relation. Thus, we may partition $[a,b]\setminus\{a,b\}$ into connected components: $[a,b]\setminus\{a,b\}=\bigcup \{C_i: i\leq Cp[a,b]\}$.

We claim that the cardinality of $Z'$ in implication $xy\rto Z'$ of the optimum basis for  $\Co(P)$  is fully defined by $Cp[x,y]$.

\begin{lem}\label{CoP}
Let $\Co(P)$ be a convex geometry of order convex subsets, and let $\Sigma_O$ be one of its optimum bases. For each implication $xy\rto Z$ of its canonical basis there exists unique implication $(xy \rto Z')\in \Sigma_O$, where $Z'\subseteq Z$ contains exactly one member of each connected component of $[x,y]\setminus\{x,y\}$. Moreover, each implication in $\Sigma_O$ is obtained this way from implication of the canonical basis.
\end{lem}
\begin{proof}
First, we need to show that $xy\rto Z$ follows from $xy\rto Z'$ and other implications of the new basis $\Sigma_O$. We observe that inference of $xy\rto z$, for $z \in Z$, from basis $\Sigma_O$ (or, any other basis), will include only implications $ab\rto c$, where $a,b,c\in [x,y]$. Compare with the Proposition 1 in \cite{AN12}.

We will argue by the induction on the height $k$ of $[x,y]$. There is no implications in the basis corresponding to $k=1$, i.e., when $x$ is covered by $y$. If $k=2$, i.e., $[x,y]$ contains the chains of maximum 3 elements, then $[x,y]\setminus\{x,y\}$ is an anti-chain $Z=\{z_1,\dots, z_n\}$. In particular, $Cp[x,y]=n$. We claim that $xy\rto Z$ from canonical basis will also be in every optimum basis. Indeed, $xy \rto z_i$  does not follow from $xy\rto Z\setminus\{z_i\}$, and there is no other implication $ab\rto c$ in $\Sigma_O$ with $a,b \in [x,y]$. 

Now assume that the height of $[x,y]$ is $k+1$, and, for every $[a,b]$ of height at most $k$, it is shown that $ab\rto [a,b]$ follows from $ab\rto Z'$, with some choice $Z'$ of representatives from the connected components of $[a,b]\setminus\{a,b\}$. Let $C$ be a connected component of $[x,y]\setminus\{x,y\}$.
Choose any $c \in C$. We claim that $xy\rto C$ follows from $xy\rto c$. Pick any $d \in C\setminus \{c\}$. Then one can find a sequence $c,m_1,m_2,\dots,m_p,d$, where $m_i$, $i\leq p$, are maximal or minimal elements of sub-poset on $[x,y]\setminus\{x,y\}$, and two consecutive elements of the sequence are comparable. Without loss of generality we may assume that, say, $c<m_1 > m_2 <\dots >m_p < d$. In this case, $cy\rto m_1$, $m_1x\rto m_2$, \dots, $m_py\rto d$ follow from the implications of $\Sigma_O$, by inductive hypothesis. Hence, $xy \rto d$ follows from $xy \rto c$ and other implications of $\Sigma_O$.   

Thus, having a single representative from each connected component will be enough to deduce the implication $xy\rto Z$ from the canonical basis.

It remains to note that we must have at least one representative from each connected component. Suppose no element from some connected component $C\subseteq [x,y]\setminus\{x,y\}$ is included into $Z'$.
The inference of $xy\rto c$, where $c\in C$ will require implication $ab\rto c$, where $a,b \in [x,y]$ and $\{a,b\}\not = \{x,y\}$.
W.l.o.g. assume $x<a<y$, $b=y$ and $a \not \in C$. Then $a < c$, which contradicts that connected component $C$ does not contain $a$.
\end{proof}

\begin{cor} The optimum basis of any convex geometry $Co(P)$ of order convex subsets of poset $P$ can be computed from the canonical basis $\Sigma_C$ in time polynomial in $s(\Sigma)$.
\end{cor}

Indeed, the claim follows from Lemma \ref{CoP}, and the observation that computation of connected components of each sub-poset $[x,y]\setminus \{x,y\}$, $x<y$, $x,y\in P$, will require the polynomial time of $s(\Sigma)$.

\section{Other convex geometries with the tractable optimum bases}

The $CQ$-closure systems in \cite{BCKK} give another example of tractable case, and this class has non-empty intersection with the class of convex geometries. For example, convex geometry given by the canonical basis
$\Sigma_C=\{a_1a_2a_3\rto xyz,a_1a_2x\rto y, a_2a_3y\rto x\}$ is  $CQ$, because $\{x,y\}$ is the only non-trivial component, and every element in the conclusion has maximum one element from its component in the premise. On the other hand, this system has a $D$-cycle $xDyDx$, and it does not satisfy the Carousel property, since $z \not \in \phi(\{x\}\cup A')$, for any $A'\subset \{a_1, a_2, a_3\}$. It is apparently not the convex geometry of order convex sets, since the premises have more than two elements.

Note that there are two optimum bases for this system: in one the first implication of $\Sigma_C$ is replaced by $a_1a_2a_3\rto xz$, and in second by $a_1a_2a_3\rto yz$. This is different from systems with Carousel property or order convex subsets, where the non-binary implications in an optimum basis have a singleton on the right side, and any singleton in the closure of the left side can be chosen for that purpose.

Still, there are convex geometries outside of all tractable subclasses discussed in this paper.

\begin{exm}
Consider convex geometry given by the canonical basis $\{a_1a_2a_3\rto xyz, a_1xy\rto z, a_2a_3z\rto y, a_2a_3y\rto x\}$.

It is not $CQ$, since one has a non-trivial component $\{x,y,z\}$, and implication $a_1xy\rto z$ includes two elements from it in the premise. It also has $D$-cycles and it does not satisfy the Carousel rule: $x \not \in \phi(\{y\} \cup A')$, for any $A'\subset \{a_1,a_2,a_3\}$. Evidently, this convex geometry cannot be $Co(P)$, since the size of left sides of implications is greater than $2$.
\end{exm}

At the moment we are not aware of any subclass of convex geometries for which optimum basis is not tractable.
So the following problem is of importance:

\begin{pb}
Determine whether there exists a polynomial algorithm of obtaining the optimum basis from a canonical basis of arbitrary convex geometry.
\end{pb}

Another related question is a recognition problem.

\begin{pb}
Given any set $\Sigma$ of implication determine whether the closure system defined by $\Sigma$ is a convex geometry. Does there exist such algorithm that performs in time polynomially dependable on $s(\Sigma)$?
\end{pb}

While it is easy to design an algorithm based on the anti-exchange property of related closure operator, such as in \cite{YHM15}, the existence of a \emph{polynomial} algorithm remains to be an intriguing question.

\emph{Acknowledgments.}  The work on the paper was inspired by the close communication with J.B.~Nation and V.~ Duquenne. We are grateful to Laurent Beaudau who pointed to the flaw in original argument about optimum basis of geometries of convex sub-posets. The author was partially supported by AWM-NSF Mentor Travel grant N0839954. The comments of anonymous referees helped to improve the presentation of results.

\end{document}